\documentclass[a4paper]{article}
\usepackage{amsmath}
\usepackage{mathrsfs}
\usepackage{amsfonts}
\usepackage{graphicx}
\usepackage{amsmath}
\usepackage{amssymb}
\usepackage[all]{xypic}
\usepackage[all]{xy}
\usepackage[toc,page,title,titletoc,header]{appendix}
\usepackage{xcolor}
\usepackage{hyperref}
\usepackage[square,numbers]{natbib}
\usepackage{fancyhdr}
\usepackage[a4paper, margin=3cm]{geometry}

\newcommand{\C}{{\mathbb C}}

\newcommand{\R}{{\mathbb R}}
\newcommand{\Q}{{\mathbb Q}}
\newcommand{\Z}{{\mathbb Z}}
\newcommand{\N}{{\mathbb N}}

\renewcommand{\Re}{\mathrm{Re}}
\renewcommand{\Im}{\mathrm{Im}}
\newcommand{\SL}{\mathrm{SL}}

\newcommand{\NN}{\mathcal{N}}
\newcommand{\OO}{\mathcal{O}}

\newcommand{\height}{\mathrm{h}}

\renewcommand{\mod}{\operatorname{mod}\nolimits}

\newcommand{\End}{\operatorname{End}\nolimits}

\renewcommand{\Im}{\operatorname{Im}\nolimits}

\newtheorem{theorem}{Theorem}[section]

\newtheorem{corollary}[theorem]{Corollary}

\newtheorem{lemma}[theorem]{Lemma}

\newtheorem{proposition}[theorem]{Proposition}

\newenvironment{proof}{{\textbf{Proof}}\par}{\hfill$\blacksquare$}

\begin{document}
	
	\title
	{Bounding the difference of two singular moduli}
	
	\author{Yulin Cai}
	
	\newcommand{\address}{{
			\bigskip
			\footnotesize
			Y.~Cai,
			\textsc{Institut de Math\'ematiques de Bordeaux, Universit\'e de Bordeaux 
				351, cours de la Lib\'eration 33405 Talence Cedex, France}\par\nopagebreak
			
			\textit{E-mail address}: \texttt{ylcai5388339@gmail.com}
	}}

	\maketitle
	
	\begin{abstract}
		For a fixed singular modulus $\alpha$, we give an effective lower bound of norm of $x-\alpha$ for another singular modulus $x$ with large discriminant. We then generalize this result for $\Phi_m(x,\alpha)$, where $\Phi_m(X,Y) \in \Z[X,Y]$ is the $m$-th modular polynomial.
	\end{abstract}
	
	{\footnotesize
		
		\tableofcontents
		
	}

	\section{Introduction}
	
	Let $\mathbb{H}$ be the upper half plane, a point $\tau \in \mathbb{H}$ is called a CM-point if $\End(E_\tau)$ is an order in an imaginary quadratic field, where $E_\tau$ is the elliptic curve over $\C$ corresponding to $\tau$. It is well-known that $\tau \in \mathbb{H}$ is CM if and only if $\tau$ is algebraic number of degree $2$. 
	We call $j(\tau)$ a singular modulus if $\tau \in \mathbb{H}$ is CM. From the classical CM-theory, we know that every singular modulus is an algebraic integer. We call  $j(\tau)$ a singular unit if it is a singular modulus and an algebraic unit.
	
	In \cite{habegger2015singular}, Habegger proved that there is at most finitely many singular units. However his proof is ineffective. After this, in \cite{bilu2017no}, Bilu, Habegger and K\"uhne  proved that there are no singular units. Indeed, their method can be generalized to give an effective bound of norm of difference between two singular moduli, that is exactly what we do in this paper. 
	
	Before stating our result, let us fix some notations. Given a singular modulus $x=j(\tau)$, the discriminant of $x$ is defined to be the discriminant of the order $\End(E_\tau)$ in an imaginary quadratic number field. We know that, in this case, $\End(E_\tau)$ is isomorphic to $\OO_\Delta: = \Z[(\Delta+\sqrt{\Delta})/2]$. By the CM-theory, the singular moduli of a given discriminant form a Galois orbit over $\Q$ of cardinality equal to the class number $\mathcal{C}(\Delta)$ of $\OO_\Delta$. For a number field $K$, and $\alpha\in K$, we denote $\NN_{K/\Q}(\alpha)$ the absolute norm of $\alpha$, and denote $\height(\alpha)$ the absolute height of $\alpha$.

	In this paper, we are going to prove the following result:
	
	\begin{theorem}
		Let $\alpha, x$ be two singular moduli of discriminants $\Delta_\alpha, \Delta$ respectively, and $K = \Q(\alpha,x)$. 
		\begin{enumerate}
			\item [(1)]
			If $\Delta_\alpha \not = -3, -4$ and $|\Delta| \geq \max\{e^{3.12}(\mathcal{C}(\Delta_\alpha)|\Delta_\alpha|^4e^{\height(\alpha)})^{3}, 10^{15}\cdot\mathcal{C}(\Delta_{\alpha})^6\},$
			then
			$$\log|\NN_{K/\mathbb{Q}}(x- \alpha)| > \dfrac{|\Delta|^{1/2}}{2};$$
			\item[(2)]
			If $\Delta_\alpha = -4$, i.e. $\alpha = 1728$, and $|\Delta| \geq 10^{15}$, then
			$$\log|\NN_{K/\mathbb{Q}}(x- 1728)| > \dfrac{2|\Delta|^{1/2}}{5};$$
			\item[(3)]
			If $\Delta_\alpha = -3$, i.e. $\alpha = 0$, and $|\Delta| \geq 10^{15}$, then
			$$\log|\NN_{K/\mathbb{Q}}(x)| > \dfrac{|\Delta|^{1/2}}{20}.$$
		\end{enumerate}
		\label{main theorem}
	\end{theorem}
	In this theorem, the bound is effective.
	
	The idea of proving Theorem \ref{main theorem} is from \cite{bilu2017no}. Set $\zeta_3 = e^{2\pi i/3}$ and $\zeta_6 = e^{\pi i/3}$, let $\mathcal{F}$ be the standard fundamental domain in the Poincar\'e plane, that is, the open hyperbolic triangle with vertices $\zeta_3 , \zeta_6$ , and $i\infty$, together with the geodesics $[i, \zeta_6]$ and $[\zeta_6 , i\infty)$. Given $\varepsilon \in (0,1/4)$, and a point $\tau\in \mathcal{F}$, denote by $\mathcal{C}_\varepsilon(\tau,\Delta)$ the number of singular moduli of discriminant $\Delta$ which can be written $j(z)$ where $z\in\mathbb{H}$ satisfies $|z-\tau|<\varepsilon$. Firstly, we give an effective upper bound of $\mathcal{C}_\varepsilon(\tau, \Delta)$, see Corollary~\ref{bound of C2}. Then by using this bound and the lower bound for the difference of two singular moduli from \cite{bilu2019separating}, we manage to give an upper bound for the height of difference, see Corollary~\ref{upperbound3} in Section~\ref{section upperbound}. The lower bound for height of difference comes from \cite{bilu2017no}, see Section~\ref{section lowerbound}. With these two bounds, by estimating each term in the both sides, we deduce Theorem \ref{main theorem}, see Section~\ref{section proof1}, \ref{section proof2}, \ref{section proof3}. 
	
	Let us remark, since Bilu, Habegger and K\"uhne \cite{bilu2017no} have given most of results we need for the case where $\tau = \zeta_6$, i.e. $\Delta_\alpha = -3$ in Theorem \ref{main theorem} (3), we will use their result directly and focus mainly on the case where $\tau \not = \zeta_6$.
	
	There are other works about the norm of difference between two singular moduli. In fact, Gross and Zagier \cite{gross1984on} stated explicit formula for absolute norm of difference between two singular moduli. With their works, Li \cite{li2018singular} also managed to give a bound of norm of difference between two singular moduli, his bound is a strictly positive number, which allows him to prove a generalized version of the main result of Bilu, Habegger and K\"uhne \cite{bilu2017no}. Even more, he gave a bound for $\log|\NN_{K/\Q}(\Phi_m(x,\alpha))|$, where $\Phi_m(X,Y) \in \Z[X,Y]$ is the $m$-th modular polynomial. However, it is not clear how his bound behaves as $\Delta\to -\infty$.
	
	We can generalize our main result to give a bound for $\NN_{K/\Q}(\Phi_m(x,\alpha))$ when the discriminant $\Delta$ of $x$ is sufficiently large. Recall the definition of $\Phi_m(X,Y)$. For $z_1, z_2\in \mathbb{H}$, 
	\[\Phi_m(j(z_1),j(z_2)) = \prod\limits_{\gamma\in \SL_2(\Z)\backslash D_m}(j(z_1) - j(\gamma z_2)),\]
	where
	\[D_m: = \left\{\left(\begin{matrix}
		a & b\\
		c & d\\
	\end{matrix}\right) \in \mathrm{M}_2(\Z)\mid ad-bc = m\right\}.\]

We have the following corollary from Theorem~\ref{main theorem}.

\begin{corollary}
	Keep the notations in Theorem~\ref{main theorem}. Let $m \geq 1$ be an integer. If $\Delta$ is sufficiently large (in terms of $\alpha$ and $m$), then 
	\[\log|\NN_{K/\Q}(\Phi_m(x,\alpha))| > \frac{|\Delta|^{1/2}}{20}.\]
	\label{bound for modular polynomials}
\end{corollary}

Our result requires to fix a singular modulus $\alpha$,  one generalization of this work is to give an explicite lower bound for $\log|\NN_{K/\Q}(x-\alpha)|$ when both $\Delta_\alpha$ and $\Delta$ vary.

Another natural generalization is to give a non-trivial lower bound for $v_p(\NN_{K/\Q}(x-\alpha))$ in terms of their discriminants, where $v_p$ is the $p$-adic discrete valuation. The motivation for this is to find all singular $S$-units (that is, singular moduli that are $S$-units). Recall
that, given a finite set $S$ of prime numbers, an $S$-unit is an algebraic number whose denominator and numerator are composed of prime ideals dividing primes from $S$. Recently, Campagna \cite{campagna2020singular} showed that, if $S$ is the set of rational primes congruent to $1$ modulo $3$, then there are no singular $S$-units.  Herrero, Menares and Rivera-Letelier \cite{herrero2021there} proved that, given a singular modulus $\alpha$, there are only finitely many singular moduli $x$ such that $x-\alpha$ is an $S$-unit. In particular, if $\alpha = 0$, there are only finitely many singular $S$-units. However, their proof is ineffective. We expect a non-trivial lower bound for $v_p(\NN_{K/\Q}(x-\alpha))$ can provide an effective method to calculate all singular moduli $x$ such that $x-\alpha$ is an $S$-unit for a given singular modulus $\alpha$ and a finite set $S$ of primes.

	\section{General setting}
	
	For a number field $K$, $x\in K$, we denote by $\NN_{K/\Q}(x)$ the absolute norm of $x$.
	
	Let $\Delta$ be a negative integer satisfying $\Delta \equiv 0, 1 \mod 4$ and 
	$$\mathcal{O}_\Delta = \Z[(\Delta + \sqrt{\Delta})/2],$$
	the imaginary quadratic order of discriminant $\Delta$. We suppose that $D$ is the discriminant of $\Q(\sqrt{\Delta})$, and $f = [\mathcal{O}_D :\mathcal{O}_\Delta]$ is the conductor of $\mathcal{O}_\Delta$, so we have $\Delta = f^2D$. We also denote the class number of the order $\mathcal{O}_\Delta$ by $\mathcal{C}(\Delta)$, since $\height$ is used for height of an algebraic number. For further uses, we define the modified conductor $\tilde{f}$ of $\mathcal{O}_\Delta$ by 
	\begin{equation*}
		\tilde{f} =  
		\begin{cases}
			f, & D \equiv 1 \mod 4,\\
			2f, & D \equiv 0 \mod 4.
		\end{cases}
	\end{equation*}
	
	On the other hand, let $\mathcal{F}$ be the standard fundamental domain in the Poincar\'e plane, that is, the
	open hyperbolic triangle with vertices $\zeta_3 , \zeta_6$ , and $i\infty$, together with the geodesics $[i, \zeta_6]$ and $[\zeta_6 , i\infty)$; here $\zeta_3 = e^{2\pi i/3}$ and $\zeta_6 = e^{\pi i/3}.$ Then
	the Klein $j$-invariant $j: \mathbb{H} \rightarrow \C$ induces a bijection
	$$j: \mathcal{F} \rightarrow \C.$$ 
	For each CM-point $\tau$ in the standard fundamental domain $\mathcal{F}$, i.e. quadratic imaginary number in $\mathcal{F}$, the discriminant $\Delta_\tau$ of $\tau$ is defined to be the discriminant of the primitive polynomial of $\tau$ over $\Z$, it is also the discriminant of the order $\End(\C/\Lambda_\tau)$, i.e. $\End(\C/\Lambda_\tau) = \mathcal{O}_{\Delta_\tau}$, where $\Lambda_\tau$ is the lattice generated by $1$ and $\tau$. Since the $j$-invariant $j: \mathcal{F} \rightarrow \C$ is a bijection, we call $\Delta_\tau$ the discriminant of $\alpha = j(\tau)$, also denoted by $\Delta_\alpha$.
	
	By classical CM-theory, we know that $\Q(\sqrt{\Delta_\tau}, j(\tau))$ is the ring class field of $\Q(\sqrt{\Delta_\tau})$ for the order $\OO_{\Delta_\tau}$, hence $\Q(\sqrt{\Delta_\tau}, j(\tau))/ \Q(\sqrt{\Delta_\tau})$ is Galois and $\mathcal{C}(\Delta_\tau) = [\Q(\sqrt{\Delta_\tau}, j(\tau)): \Q(\sqrt{\Delta_\tau})] = [\Q(j(\tau)): \Q]$.

	For $n\in \N^+$, we denote
	$$\omega(n) = \sum\limits_{p|n}1,\ \ \ \sigma_0(n) = \sum\limits_{d|n}1,\ \ \ \sigma_1(n) = \sum\limits_{d|n}d.$$

	\section{An Estimate for $C_\varepsilon(\tau,\Delta)$ \label{section bound C} }
	
	For each $\tau \in \mathcal{F}$ and $\varepsilon \in (0,1/4)$, we define 
	$$S_\varepsilon(\tau,\Delta) = \{z\in \mathbb{H} \mid z\ \textrm{is a imaginary quadratic number of discriminant}\ \Delta \ \textrm{and}\ |z - \tau| < \varepsilon \},$$
	$$\mathcal{C}_\varepsilon(\tau,\Delta) = \#S_\varepsilon(\tau,\Delta),$$	
	here $\#$ means the cardinality of a set.
	
	Let $S_\Delta$ be the set of primitive positive definite forms of discriminant $\Delta$, that is, a quadratic form $ax^2+bxy+cy^2 \in S_\Delta$ if $a, b, c \in \Z$ and
	$$a > 0, \ \  \gcd(a,b,c) = 1, \ \ \Delta = b^2 - 4ac <0$$
	For $ax^2+bxy+cy^2 \in S_\Delta$, we set 
	$$\tau(a,b,c) = \frac{b+\sqrt{\Delta}}{2a}.$$
	then the map $ax^2+bxy+cy^2 \mapsto \tau(a,b,c)$ defines  a bijection  from $S_\Delta$ to the set of imaginary number of discriminant $\Delta$.
	
	We will prove the following theorem and corollary:
	
	\begin{theorem}
		Let $\tau \in \mathcal{F}$ and $\varepsilon \in (0,1/4)$, then
		$$\mathcal{C}_\varepsilon(\tau,\Delta) \leq F\times\left(\frac{48+16\sqrt{3}}{3}\frac{\sigma_1(\tilde{f})}{\tilde{f}}|\Delta|^{1/2}\varepsilon^2 + \frac{12+4\sqrt{3}}{3}|\Delta|^{1/2}\varepsilon + \frac{8|\Delta|^{1/4}}{(\sqrt{3} -1)^{1/2}}\sigma_0(\tilde{f})\varepsilon  + 2\right),$$
		\label{bound of C}
		where 
		\begin{equation}
			F = F(\Delta) = \max\{2^{\omega(a)}: a \leq |\Delta|^{1/2}\}.
			\label{F}
		\end{equation}
	\label{general bound for C2}
	\end{theorem}
	
	\begin{corollary}
		In the set-up of Theorem \ref{bound of C}, assume that $|\Delta|\geq 10^{14}$. Then
		$$\mathcal{C}_\varepsilon(\tau,\Delta) \leq F\times\left(46.488|\Delta|^{1/2}\varepsilon^2\log\log|\Delta|^{1/2} + 7.752|\Delta|^{1/2}\varepsilon  + 2\right)$$
		\label{bound of C2}
	\end{corollary}
	
	\subsection{Some lemmas}
	
	We say that $d \in \Z$ is a quadratic divisor of $n \in \Z$ if $d^2|n$. We denote by $\gcd_2(m,n)$ the greatest common quadratic divisor of integers $m$ and $n$.  
	
	We will use the following lemmas.
	
	\begin{lemma}[\cite{bilu2017no}, Lemma 2.4]
		Let $a$ be a positive integer and $\Delta$ a non-zero integer. Then the set of $b \in \Z$ satisfying $b^2 \equiv \Delta \mod a$ consists of at most $2^{\omega(a/\gcd(a,\Delta))+1}$ residue classes modulo $a/\gcd_2(a,\Delta)$.
		\label{number of roots}
	\end{lemma}
	
	\begin{lemma}
		Let $\alpha,\beta \in \R$ be such that $\alpha < \beta$, and $m$  a positive integer. Then every residue class modulo $m$ has at most $(\beta - \alpha)/m +1$ elements in the interval $[\alpha, \beta]$.
		\label{number in interval}
	\end{lemma}
	
	\begin{lemma}
		Let $\tau \in \mathcal{F}$, and $\varepsilon \in (0, 1/4)$, and let $ax^2+bxy+cy^2 \in S_\Delta$ be such that $|\tau(a,b,c) - \tau| < \varepsilon$. Then
		\begin{equation}
			\frac{|\Delta|^{1/2}}{2(\Im\tau + \varepsilon)} < a < \frac{|\Delta|^{1/2}}{2(\Im\tau - \varepsilon)},
			\label{bound a1}
		\end{equation}
		\begin{equation}
			2a(\Re\tau - \varepsilon) < b < 2a(\Re\tau + \varepsilon).
			\label{bound b1}
		\end{equation}		
	\end{lemma}
	\begin{proof}
		Set $z = \tau(a,b,c)$, then from $|z-\tau|< \varepsilon$, we have 
		$$|\Im z - \Im\tau| < \varepsilon,  \ \ |\Re z - \Re\tau| < \varepsilon,$$
		that is, 
		$$\left|\frac{|\Delta|^{1/2}}{2a} - \Im\tau\right| < \varepsilon,  \ \  \left|\frac{b}{2a} - \Re\tau\right| < \varepsilon,$$
		so we have (\ref{bound a1}) and (\ref{bound b1}). 
	\end{proof}
	
	\subsection{Proof of Theorem \ref{bound of C}}
	
	Set 
	$$I = \left(\frac{|\Delta|^{1/2}}{2(\Im\tau + \varepsilon)}, \frac{|\Delta|^{1/2}}{2(\Im\tau - \varepsilon)}\right),$$
	$$\tau(a,b,c) = \frac{b+\sqrt{\Delta}}{2a}.$$
	By Lemma~\ref{lemma 5.25}, if $\tau(a,b,c) \in S_\varepsilon(\tau,\Delta)$, then $a \in I$ and $b \in (2a(\Re\tau - \varepsilon), 2a(\Re\tau + \varepsilon)).$
	
	For a fixed $a$, by Lemma \ref{number of roots} and Lemma \ref{number in interval} and $\omega(a/\gcd(a,\Delta)) \leq \omega(a)$, there are at most $(4\varepsilon\gcd_2(a,\Delta) + 1)\cdot 2^{\omega(a)+1}$ possible $b$'s. Since $\varepsilon < 1/4, \Im\tau\geq \sqrt{3}/2$, then $\frac{|\Delta|^{1/2}}{2(\Im\tau - \varepsilon)} \leq |\Delta|^{1/2}$. Hence
	\begin{align*}
		\mathcal{C}_\varepsilon(\tau,\Delta) &\leq 8\varepsilon \sum\limits_{a\in I\cap\Z} \gcd{_2}(a,\Delta)\cdot 2^{\omega(a)} + 2\sum\limits_{a\in I\cap\Z}2^{\omega(a)}\\
		&\leq 8\varepsilon F\sum\limits_{a\in I\cap\Z}\gcd{_2}(a,\Delta) + 2F\#(I\cap\Z).
	\end{align*}
	Note that 
	$$\sum\limits_{a\in I\cap\Z}\gcd{_2}(a,\Delta) \leq \sum\limits_{d^2\mid\Delta}d\cdot \#(I\cap d^2\Z),$$
	and the length of $I$ is
	\begin{align*}
		\frac{|\Delta|^{1/2}}{2(\Im\tau - \varepsilon)} - \frac{|\Delta|^{1/2}}{2(\Im\tau + \varepsilon)}  & =|\Delta|^{1/2}\frac{\varepsilon}{(\Im\tau+\varepsilon)(\Im\tau - \varepsilon)}\\
		&\leq |\Delta|^{1/2}\frac{\varepsilon}{\sqrt{3}/2(\sqrt{3}/2 -1/2)}\\
		&= \frac{6+2\sqrt{3}}{3}|\Delta|^{1/2}\varepsilon.
	\end{align*}
	When $d > \frac{|\Delta|^{1/4}}{(\sqrt{3} -1)^{1/2}}$, we have $\frac{|\Delta|^{1/2}}{2(\Im\tau - \varepsilon)} < d^2$. Combine this with Lemma \ref{number in interval}, we have
	\begin{equation*}
		\#(I\cap d^2\Z) \leq 
		\begin{cases}
			\frac{6+2\sqrt{3}}{3}\frac{|\Delta|^{1/2}}{d^2}\varepsilon +1 & d\leq \frac{|\Delta|^{1/4}}{(\sqrt{3} -1)^{1/2}},\\
			0 & d > \frac{|\Delta|^{1/4}}{(\sqrt{3} -1)^{1/2}}.
		\end{cases}
	\end{equation*}
	Since $\Delta/\tilde{f}^2$ is square-free, so for a positive integer $d$, $d^2\mid\Delta$ if and only if $d\mid\tilde{f}$, hence
	\begin{align*}
		\sum\limits_{d^2\mid\Delta}d\cdot \#(I\cap d^2\Z) & \leq \sum\limits_{\substack{d\mid\tilde{f} \\ d\leq \frac{|\Delta|^{1/4}}{(\sqrt{3} -1)^{1/2}}}}d\left(\frac{6+2\sqrt{3}}{3}\frac{|\Delta|^{1/2}}{d^2}\varepsilon +1\right)\\
		&\leq \frac{6+2\sqrt{3}}{3}|\Delta|^{1/2}\varepsilon\sum\limits_{d\mid\tilde{f}}1/d + \sum\limits_{\substack{d\mid\tilde{f} \\ d\leq \frac{|\Delta|^{1/4}}{(\sqrt{3} -1)^{1/2}}}}d\\
		&\leq \frac{6+2\sqrt{3}}{3}\frac{\sigma_1(\tilde{f})}{\tilde{f}}|\Delta|^{1/2}\varepsilon + \frac{|\Delta|^{1/4}}{(\sqrt{3} -1)^{1/2}}\sigma_0(\tilde{f}).
	\end{align*}
	Again, by Lemma \ref{number in interval}, we have
	$$\#(I\cap\Z) \leq \frac{6+2\sqrt{3}}{3}|\Delta|^{1/2}\varepsilon + 1.$$
	Hence, 
	\begin{align*}
		\mathcal{C}_\varepsilon(\tau,\Delta) &\leq 8\varepsilon F\times\left(\frac{6+2\sqrt{3}}{3}\frac{\sigma_1(\tilde{f})}{\tilde{f}}|\Delta|^{1/2}\varepsilon + \frac{|\Delta|^{1/4}}{(\sqrt{3} -1)^{1/2}}\sigma_0(\tilde{f})\right) + 2F\times\left(\frac{6+2\sqrt{3}}{3}|\Delta|^{1/2}\varepsilon + 1\right)\\
		&\leq F\times\left(\frac{48+16\sqrt{3}}{3}\frac{\sigma_1(\tilde{f})}{\tilde{f}}|\Delta|^{1/2}\varepsilon^2 + \frac{12+4\sqrt{3}}{3}|\Delta|^{1/2}\varepsilon + \frac{8|\Delta|^{1/4}}{(\sqrt{3} -1)^{1/2}}\sigma_0(\tilde{f})\varepsilon  + 2\right).
	\end{align*} 
	
	\subsection{Proof of Corollary \ref{bound of C2}}
	The following lemma estimate $\sigma_0(\tilde{f})$ and  $\sigma_1(\tilde{f})$ in terms of $|\Delta|$:
	
	\begin{lemma}[\cite{bilu2017no}, Lemma 2.8]
		For $|\Delta|\geq 10^{14}$, we have
		\begin{equation*}
			\sigma_0(\tilde{f}) \leq |\Delta|^{0.192},
		\end{equation*}
		\begin{equation*}
			\sigma_1(\tilde{f})/\tilde{f} \leq 1.842\log\log|\Delta|^{1/2}.
		\end{equation*}
	\end{lemma}

	With this lemma, we have 
	$$\frac{48+16\sqrt{3}}{3}\frac{\sigma_1(\tilde{f})}{\tilde{f}} \leq \frac{48+16\sqrt{3}}{3} \cdot 1.842\log\log|\Delta|^{1/2} \leq 46.488\log\log|\Delta|^{1/2},$$
	$$\frac{8|\Delta|^{1/4}}{(\sqrt{3} -1)^{1/2}}\sigma_0(\tilde{f}) \leq \frac{8}{(\sqrt{3} -1)^{1/2}}|\Delta|^{0.442} \leq \frac{8}{(\sqrt{3} -1)^{1/2}\cdot 10^{0.812}}|\Delta|^{1/2}\leq 1.442|\Delta|^{1/2}.$$
	$$\frac{12+4\sqrt{3}}{3} + 1.442 \leq 7.752$$
	With these bounds and Theorem \ref{bound of C}, we have Corollary \ref{bound of C2}.
	
	\section{An Upper Bound for the Height of the difference of Singular Moduli \label{section upperbound}}
	
	Let $\alpha = j(\tau), x = j(z)$ be two different singular moduli with $\tau, z \in \mathcal{F}$, and $\Delta_\alpha$, $\Delta=\Delta_x$ be their discriminants respectively. Let $K = \Q(x-\alpha)$, $d = [K : \Q]$, then we have $K = \Q(\alpha, x)$, see \cite[Theorem 4.1]{faye2018fields}. Hence we can assume that $d = s\mathcal{C}(\Delta_\alpha)$, where $s = [K : \Q(\alpha)]$. 
	We suppose that the set of embeddings of $K$ to $\C$ is $\{\sigma_1, \cdots, \sigma_d\}$. For each $k$, set $\alpha_k = \sigma_k(\alpha) = j(\tau_k) $ with $\tau_k\in \mathcal{F}$,  and set $x_k = \sigma_k(x) = j(z_k)$ such that $z_k \in \mathbb{H}$ is the nearest point to $\tau_k$ among $\SL_2(\Z)z_k$ with respect to the absolute valuation on $\C$.  Then $\alpha_k \not= x_k$ for each $k$, and we have
	\begin{equation}
		\height(x - \alpha) = \height((x-\alpha)^{-1}) = \dfrac{1}{d}\sum\limits_{k=1}^{d}\log^+|x_k - \alpha_k|^{-1} + \frac{1}{d}\log|\NN_{K/\Q}(x-\alpha)|,
		\label{height}
	\end{equation}
	where $\log^+(\cdot) = \log\max\{1,\cdot\}$. 
	
	In this section, we are going to prove that following theorem and corollary:
	\begin{theorem}
		Let $\alpha = j(\tau), x = j(z)$ be two different singular moduli with $\tau, z \in \mathcal{F}$, and $\Delta_\alpha$, $\Delta=\Delta_x$  their discriminants respectively. Let $K = \Q(x-\alpha)$, $d = [K : \Q]$,
		\begin{enumerate}
			\item[(1)]	
			if $\tau \not= i, \zeta_6$ and $0<\varepsilon< \min\{\frac{1}{3|\Delta_\alpha|^2}, 10^{-8}\}$, then
			\begin{align*}
				\height(x - \alpha) &\leq \sum\limits_{1 \leq k  \leq \mathcal{C}(\Delta_\alpha)}4\frac{\mathcal{C}_\varepsilon(\tau_k,\Delta)}{d}\log(\max\{|\Delta|, |\Delta_\alpha|\}) + \log(\varepsilon^{-1}) + 2\log|\Delta_\alpha|- 7.783 \\
				&\ \ \  + \frac{1}{d}\log|\NN_{K/\Q}(x-\alpha)|;
			\end{align*}
			\item[(2)] if $\tau = i$ and $0< \varepsilon \leq  7\cdot 10^{-3}$, then
			$$\height(x-1728) \leq 2\frac{\mathcal{C}_\varepsilon(i,\Delta)}{\mathcal{C}(\Delta)}\log|\Delta| + 2\log\varepsilon^{-1} - 9.9 + \frac{1}{\mathcal{C}(\Delta)}\log|\NN_{K/\Q}(x-1728)|.$$
		\end{enumerate}
		\label{upperbound}
	\end{theorem}
	
	We don't discuss the case where $\tau = \zeta_6$, since the bound for this case in the following corollary can be get directly from \cite{bilu2017no}. 
	
	\begin{corollary}
		In the setup of Theorem \ref{upperbound}, assume that $|\Delta| \geq 10^{14}$,
		\begin{enumerate}
			\item[(1)]
			if $\tau \not= i, \zeta_6$, then
			$$\height(x - \alpha) \leq \frac{8A\mathcal{C}(\Delta_\alpha)}{d} + \log(\frac{A\mathcal{C}(\Delta_\alpha)|\Delta|^{1/2}}{d})  + 4\log|\Delta_\alpha| + 0.33 +\frac{1}{d}\log|\NN_{K/\Q}(x-\alpha)|;$$
			\item[(2)] if $\tau = i$, then
			$$\height(x-1728) \leq  \frac{4A}{\mathcal{C}(\Delta)} + 2\log\frac{A|\Delta|^{1/2}}{\mathcal{C}(\Delta)} -2.68 + \frac{1}{\mathcal{C}(\Delta)}\log|\NN_{K/\Q}(x-1728)|;$$
			\item[(3)]if $\tau = \zeta_6$, then
			$$\height(x) \leq \dfrac{12A}{\mathcal{C}(\Delta)} + 3\log\frac{A|\Delta|^{1/2}}{\mathcal{C}(\Delta)} - 3.77 + \frac{1}{\mathcal{C}(\Delta)}\log|\NN_{K/\Q}(x)|,$$
		\end{enumerate}
		where $A= F\log\max\{|\Delta|,|\Delta_\alpha|\}$ and $F$ is defined in Theorem~\ref{general bound for C2}.
		\label{upperbound3}
	\end{corollary}

	\subsection{Proof of Theorem \ref{upperbound}}
	
	The following lemmas and theorems are needed.
	
	\begin{lemma}
		In the set-up of Theorem \ref{upperbound},
		\begin{enumerate}
			\item [1)] if $\Im\tau \geq 1.3$, then there exists $z' \in \mathbb{H}$ with $x = j(z')$ such that
			$$|x - \alpha| \geq e^{2.6\pi}\min\{0.4 |z' - \tau|, 0.04\};$$
			\item [2)] if $\Im\tau \leq 1.3$ and $\tau \not = i, \zeta_6$, then there exist $z' \in \mathbb{H}$ with $x = j(z')$ such that
			$$|x - \alpha|  \geq \min\{5\cdot10^{-7}, 800|\Delta_\alpha|^{-4}, 2400|\Delta_\alpha|^{-2}|z'-\tau|\}.$$
			\label{4.3inequality}
		\end{enumerate}
	\end{lemma}
	\begin{proof}
		Combine Proposition 4.1 and Proposition 4.2 in \cite{bilu2019separating}.
	\end{proof}

	\begin{theorem}[\cite{bilu2019separating} Theorem 1.1]
		In the set-up of Theorem \ref{upperbound}, we have
		$$|x-\alpha| \geq 800\max\{|\Delta|,|\Delta_\alpha|\}^{-4}.$$
		\label{Yuri main theorem}
	\end{theorem}
	
	\begin{lemma}
		For $i \not = z\in \mathcal{F}$ with discriminant $\Delta$, we have
		$$|j(z) - 1728| \geq 20000\min\{|z-i|, 0.01\}^2,$$
		$$|j(z) - 1728| \geq 2000|\Delta|^{-2}.$$
		\label{inequality for i}
	\end{lemma}
	\begin{proof}
		Combine Proposition 3.7 and Corollary 5.3 in \cite{bilu2019separating}.
	\end{proof}
	
	We start to prove Theorem \ref{upperbound} (1). Let $\tau_k, z_k, \alpha_k, x_k$ be as the begining of this section. Then we have
	$$\sum\limits_{k=1}^{d}\log^+|x_k - \alpha_k|^{-1} = \sum\limits_{\substack{1 \leq k \leq d \\ z_k\in S_\varepsilon(\tau_k,\Delta)}}\log^+|x_k - \alpha_k|^{-1} + \sum\limits_{\substack{1 \leq k \leq d \\ z_k\not\in S_\varepsilon(\tau_k,\Delta)}}\log^+|x_k - \alpha_k|^{-1}$$
	
	For the first sum, by Theorem \ref{Yuri main theorem}, each term in the sum has
	$$\log^+|x_k - \alpha_k|^{-1} \leq \max\{0, 4\log(\max\{|\Delta|,|\Delta_\alpha|\}) - \log(800)\} \leq 4\log(\max\{|\Delta|,|\Delta_\alpha|\}),$$
	so we have
	\begin{equation}
		\sum\limits_{\substack{1 \leq k \leq d \\ z_k\in S_\varepsilon(\tau_k,\Delta)}}\log^+|x_k - \alpha_k|^{-1} \leq \sum\limits_{1 \leq k  \leq \mathcal{C}(\Delta_\alpha)}4\mathcal{C}_\varepsilon(\tau_k,\Delta)\log(\max\{|\Delta|, |\Delta_\alpha|\}).
		\label{4.1inequality1}
	\end{equation}
	
	For the second sum,  we claim that if $|z_k - \tau_k|\geq \varepsilon$, then
	$$|x_k-\alpha_k| \geq 2400|\Delta_\alpha|^{-2}\varepsilon.$$
	In fact, we can replace $\tau$ by $\tau_k$ and $z'$ by $z_k$ in Lemma \ref{4.3inequality} by the choice of $z_k$, i.e. $z_k$ is the nearest point to $x_k$ among $\SL_2(\Z)z_k\subset \mathbb{H}$ with respect to the absolute valuation, then 
	$$|x_k-\alpha_k| \geq \min\{e^{2.6\pi} \cdot 0.4 \varepsilon, 5\cdot 10^{-7}, 800|\Delta_\alpha|^{-4}, 2400|\Delta_\alpha|^{-2}\varepsilon\}.$$
	Notice that $|\Delta_\alpha| \geq 7$ and  $\varepsilon< \min\{\frac{1}{3|\Delta_\alpha|^2}, 10^{-8}\}$, then
	$$2400|\Delta_\alpha|^{-2}\varepsilon \leq 800|\Delta_\alpha|^{-4},$$
	$$2400|\Delta_\alpha|^{-2}\varepsilon \leq \frac{2400}{49}  \cdot 10^{-8} < 5\cdot 10^{-7},$$ 
	$$2400|\Delta_\alpha|^{-2}\varepsilon \leq \frac{2400}{49} \varepsilon \leq 1410\varepsilon \leq e^{2.6\pi} \cdot 0.4\varepsilon,$$
	so we have our claim. Hence
	$$\log^+|x_k - \alpha_k|^{-1} \leq \log\left(\frac{|\Delta_\alpha|^2}{2400}\varepsilon^{-1}\right) \leq  \log(\varepsilon^{-1}) + 2\log|\Delta_\alpha|- 7.783,$$
	\begin{equation}
		\sum\limits_{\substack{1 \leq k \leq d \\ z_k\not\in S_\varepsilon(\tau_k,\Delta)}}\log^+|x_k - \alpha_k|^{-1} \leq d(\log(\varepsilon^{-1}) + 2\log|\Delta_\alpha|- 7.783).
		\label{4.1inequality2}
	\end{equation}
	Combine  (\ref{4.1inequality1}), (\ref{4.1inequality2}) and the equality (\ref{height}), we have the bound in (1).
	
	For Theorem~\ref{main theorem} (2), the proof is similar as above. Since $j(\tau) = 1728$, then $d = \mathcal{C}(\Delta)$ and 
	$$\sum\limits_{k=1}^{\mathcal{C}(\Delta)}\log^+|x_k - 1728|^{-1} = \sum\limits_{\substack{1 \leq k \leq \mathcal{C}(\Delta) \\ z_k\in S_\varepsilon(i,\Delta)}}\log^+|x_k - 1728|^{-1} + \sum\limits_{\substack{1 \leq k \leq \mathcal{C}(\Delta) \\ z_k\not\in S_\varepsilon(i,\Delta)}}\log^+|x_k - 1728|^{-1} $$
	
	For the first sum, by Lemma \ref{inequality for i},
	$$\log^+|x_k - 1728|^{-1} \leq \max\{0, 2\log|\Delta| - \log 2000\} \leq 2\log|\Delta|,$$
	$$\sum\limits_{\substack{1 \leq k \leq \mathcal{C}(\Delta) \\ z_k\in S_\varepsilon(i,\Delta)}}\log^+|x_k - 1728|^{-1}\leq 2\mathcal{C}_\varepsilon(i, \Delta)\log|\Delta|.$$
	
	For the second sum, since $\varepsilon \leq 7\cdot 10^{-3}$, $\varepsilon^{-2} > 20000$ and $|z_k - i| \geq \varepsilon$, we have
	$$|x_k - 1728|^{-1} \leq 20000^{-1}\min\{\varepsilon, 0.01\}^{-2} = 20000^{-1}\varepsilon^{-2},$$
	$$\log^+|x_k - 1728|^{-1} \leq \max\{0, 2\log\varepsilon^{-1} - \log(20000)\} \leq 2\log\varepsilon^{-1} - 9.9,$$
	$$\sum\limits_{\substack{1 \leq k \leq \mathcal{C}(\Delta) \\ z_k\not\in S_\varepsilon(i,\Delta)}}\log^+|x_k - 1728|^{-1} \leq \mathcal{C}(\Delta)(2\log\varepsilon^{-1} - 9.9).$$
	Hence, as above, we have 
	$$\height(x-1728) \leq 2\frac{\mathcal{C}_\varepsilon(i,\Delta)}{\mathcal{C}(\Delta)}\log|\Delta| + 2\log\varepsilon^{-1} - 9.9 + \frac{1}{\mathcal{C}(\Delta)}\log|\NN_{K/\Q}(x-1728)|.$$
	\subsection{Proof of Corollary \ref{upperbound3}}
	
	We will use the following lemmas from \cite{bilu2017no}.
	
	\begin{lemma}[\cite{bilu2017no} Lemma 3.5]
		Assume that $|\Delta| \geq 10^{14}$. Then $F \geq |\Delta|^{0.34/\log\log(|\Delta|^{1/2})}$ and 
		$F \geq 18.54\log\log(|\Delta|^{1/2}).$
		\label{yurilemma1}
	\end{lemma}
	
	\begin{lemma}[\cite{bilu2017no} Lemma 3.6]
		For $\Delta \not = -3,-4$, we have
		$$\mathcal{C}(\Delta) \leq \pi^{-1}|\Delta|^{1/2}(2+\log|\Delta|).$$
		\label{yurilemma2}
	\end{lemma}
	
	To prove  Corollary~\ref{upperbound3} (1), by Corollary \ref{bound of C2}, we have
	\begin{equation}
		\sum\limits_{1 \leq k  \leq \mathcal{C}(\Delta_\alpha)}4\frac{\mathcal{C}_\varepsilon(\tau_k,\Delta)}{d}\log\max\{|\Delta|, |\Delta_\alpha|\} \leq 4\frac{A\mathcal{C}(\Delta_\alpha) \left(46.488|\Delta|^{1/2}\varepsilon^2\log\log|\Delta|^{1/2} + 7.752|\Delta|^{1/2}\varepsilon  + 2\right)}{d}.
		\label{inequality3}
	\end{equation}
	
	We can take $\varepsilon = 0.0003\frac{d}{A\mathcal{C}(\Delta_\alpha)|\Delta|^{1/2}|\Delta_\alpha|^2}$, then 
	$\varepsilon \leq \min\{\frac{1}{3|\Delta_\alpha|^2}, 10^{-8}\}$. Indeed, $F \geq 256$ if $|\Delta| \geq 10^{14}$, and by Lemma \ref{yurilemma1} and Lemma \ref{yurilemma2}, we have
	$$0.0003\frac{d}{A\mathcal{C}(\Delta_\alpha)|\Delta|^{1/2}} \leq \frac{3\mathcal{C}(\Delta)}{10000F|\Delta|^{1/2}\log|\Delta|} \leq \frac{6+3\log(10^{14})}{10000\pi\log(10^{14})} \cdot \frac{1}{256} \leq \frac{1}{3},$$
	$$0.0003\frac{d}{A\mathcal{C}(\Delta_\alpha)|\Delta|^{1/2}|\Delta_\alpha|^2} \leq \frac{6+3\log(10^{14})}{490000\pi\log(10^{14})}\cdot \frac{1}{256} \leq 10^{-8}.$$
	
	We estimate each term in the left of (\ref{inequality3}) with our $\varepsilon$:
	\begin{align*}
		4\frac{46.488A\mathcal{C}(\Delta_{\alpha})|\Delta|^{1/2}\varepsilon^2\log\log|\Delta|^{1/2}}{d}& \leq  
		36\cdot 10^{-8} \cdot 46.488\frac{d\log\log|\Delta|^{1/2}}{A\mathcal{C}(\Delta_{\alpha})|\Delta|^{1/2}|\Delta_\alpha|^4}\\
		&\leq \frac{36\cdot 10^{-8} \cdot 46.488}{|\Delta_\alpha|^4} \frac{\log\log|\Delta|^{1/2}}{F} \frac{\mathcal{C}(\Delta)}{|\Delta|^{1/2}\log|\Delta|}\\
		&\leq \frac{36 \cdot 10^{-8} \cdot 48.488 \cdot (2+ \log(10^{14}))}{18.54 \cdot \pi\log(10^{14})} \cdot \frac{1}{7^4}\\
		&\leq 0.0005;
	\end{align*}
	\begin{align*}
		4\frac{7.752A\mathcal{C}(\Delta_{\alpha})|\Delta|^{1/2}\varepsilon}{d}& \leq  
		0.0003\cdot 31.008|\Delta_\alpha|^{-2}\\
		&\leq 0.0005.
	\end{align*}
	With above, we have
	\begin{align*}
		\height(x-\alpha) &\leq \frac{8A\mathcal{C}(\Delta_\alpha)}{d} + \log(\frac{A\mathcal{C}(\Delta_\alpha)|\Delta|^{1/2}|\Delta_\alpha|^2}{d})  + 2\log|\Delta_\alpha| + 0.001 + \log(\frac{10000}{3}) - 7.783\\ & \ \  \ +\frac{1}{d}\log|\NN_{K/\Q}(x-\alpha)|\\
		& \leq \frac{8A\mathcal{C}(\Delta_\alpha)}{d} + \log(\frac{A\mathcal{C}(\Delta_\alpha)|\Delta|^{1/2}}{d})  + 4\log|\Delta_\alpha| + 0.33 +\frac{1}{d}\log|\NN_{K/\Q}(x-\alpha)|.
	\end{align*}
	
	For  Corollary~\ref{upperbound3} (2), the proof is similar. We set $\varepsilon = 0.3\frac{\mathcal{C}(\Delta)}{A|\Delta|^{1/2}}$, then $\varepsilon \leq 7\cdot 10^{-3}$. Indeed, since $|\Delta| \geq  10^{14}$, so $F \geq 256$, hence
	$$0.3\frac{\mathcal{C}(\Delta)}{A|\Delta|^{1/2}}  \leq 0.3 \frac{\mathcal{C}(\Delta)}{|\Delta|^{1/2}\log|\Delta|}\cdot \frac{1}{F} \leq 0.3\frac{2+\log(10^{14})}{\pi\log(10^{14})} \cdot \frac{1}{256} \leq 5\cdot 10^{-4}.$$
	By Corollary \ref{bound of C2}, Theorem \ref{upperbound}(2), Lemma \ref{yurilemma1} and Lemma \ref{yurilemma2}, we have
	\begin{align*}
		\height(x - 1728) &\leq 2\frac{\mathcal{C}_\varepsilon(i,\Delta)}{\mathcal{C}(\Delta)}\log|\Delta| + 2\log\varepsilon^{-1} - 9.9 + \frac{1}{\mathcal{C}(\Delta)}\log|\NN_{K/\Q}(x-1728)|\\
		&\leq 2\frac{A\left(46.488|\Delta|^{1/2}\varepsilon^2\log\log|\Delta|^{1/2} + 7.752|\Delta|^{1/2}\varepsilon  + 2\right)}{\mathcal{C}(\Delta)}+ 2\log\varepsilon^{-1} - 9.9\\ 
		& \ \ \ + \frac{1}{\mathcal{C}(\Delta)}\log|\NN_{K/\Q}(x-1728)|\\
		& \leq 2\cdot 46.488 \cdot 0.3^2 \frac{\log\log|\Delta|^{1/2}}{F}\frac{\mathcal{C}(\Delta)}{|\Delta|^{1/2}\log|\Delta|} + 2\cdot 0.3\cdot 7.752 + \frac{4A}{\mathcal{C}(\Delta)} \\
		& \ \ \ + 2\log\frac{A|\Delta|^{1/2}}{\mathcal{C}(\Delta)} -2\log 0.3 - 9.9 + \frac{1}{\mathcal{C}(\Delta)}\log|\NN_{K/\Q}(x-1728)|\\
		& \leq  \frac{4A}{\mathcal{C}(\Delta)} + 2\log\frac{A|\Delta|^{1/2}}{\mathcal{C}(\Delta)} + 2\cdot 46.488 \cdot 0.3^2 \frac{2+\log(10^{14})}{18.54\cdot\pi\log(10^{14})} -2.84 \\ 
		& \ \ \   + \frac{1}{\mathcal{C}(\Delta)}\log|\NN_{K/\Q}(x-1728)|\\
		& \leq  \frac{4A}{\mathcal{C}(\Delta)} + 2\log\frac{A|\Delta|^{1/2}}{\mathcal{C}(\Delta)} -2.68 + \frac{1}{\mathcal{C}(\Delta)}\log|\NN_{K/\Q}(x-1728)|.
	\end{align*} 
	
	For  Corollary~\ref{upperbound3} (3), see \cite[Corollary 3.2]{bilu2017no}, without assuming that $x$ is a singular unit, we add the term $\frac{1}{\mathcal{C}(\Delta)}\log|\NN_{K/\Q}(x)|.$
	\section{Lower Bounds for the Height of a Singular Modulus \label{section lowerbound}}
	
	We have these propositions from \cite{bilu2017no}:
	\begin{proposition}[\cite{bilu2017no} Proposition 4.1]
		Let $x$ be a singular modulus of discriminant $\Delta$. Assume that $|\Delta| \geq 16.$
		Then
		$$\height(x) \geq \dfrac{\pi|\Delta|^{1/2}-0.01}{\mathcal{C}(\Delta)}.$$
		\label{lowerbound1}
	\end{proposition}
	
	\begin{proposition}
		Let $x$ be a singular modulus of discriminant $\Delta$. 
		Then
		$$\height(x) \geq \dfrac{3}{\sqrt{5}}\log|\Delta| - 9.79;$$
		$$\height(x) \geq \dfrac{1}{4\sqrt{5}}\log|\Delta| - 5.93.$$
		\label{lowerbound2}
	\end{proposition}
	\begin{proof}
		The first one see \cite[Proposition 4.3]{bilu2017no}, the second one see \cite[Lemma 14 (ii)]{habegger2017six}
	\end{proof}
	
	We can use the inequalilty $\height(x-\alpha) \geq \height(x) - \height(\alpha) - \log 2$ and the results above to give the lower bounds of $\height(x-\alpha)$ for an fixed $\alpha$.
	
	\section{Proof of Theorem \ref{main theorem} (1) \label{section proof1}}
	
	As the set-up in Section~\ref{section upperbound}, Proposition \ref{lowerbound1} and \ref{lowerbound2} allow us to give lower bounds of the height of $x-\alpha$:
	\begin{equation}
		\height(x-\alpha) \geq \height(x) - \height(\alpha) - \log 2 \geq \frac{\pi|\Delta|^{1/2}-0.01}{\mathcal{C}(\Delta)}  - \height(\alpha) - \log 2,
		\label{lower1}
	\end{equation}
	\begin{equation}
		\height(x-\alpha) \geq \height(x) - \height(\alpha) - \log 2 \geq \frac{3}{\sqrt{5}}\log|\Delta| - \height(\alpha) -9.79 - \log 2.
		\label{lower2}
	\end{equation}
	For Theorem~\ref{main theorem} (1), recall the upper bound of $x-\alpha$ in Corollary~\ref{upperbound3} (1) when $|\Delta| \geq 10^{14}$:
	\begin{equation}
		\height(x - \alpha) \leq \frac{8A\mathcal{C}(\Delta_\tau)}{d} + \log(\frac{A\mathcal{C}(\Delta_\alpha)|\Delta|^{1/2}}{d})  + 4\log|\Delta_\alpha| +0.33 +\frac{1}{d}\log|\NN_{K/\Q}(x-\alpha)|.
		\label{upper1}
	\end{equation}
	Throughout the proof of Theorem~\ref{main theorem} (1),  denote  the discriminant of a singular modulus $x = j(z)$ by $\Delta$, and we assume that $X = |\Delta| \geq \max\{e^{3.12}(\mathcal{C}(\Delta_\alpha)|\Delta_\alpha|^4e^{\height(\alpha)})^{3}, 10^{15}\cdot\mathcal{C}(\Delta_{\alpha})^6\}$. Hence $|\Delta| \geq |\Delta_\alpha|$, since $\height(\alpha) \geq 0$.

	\subsection{The main inequality}
	
	Recall that $A = F\log\max\{|\Delta|, |\Delta_\alpha|\} = F\log X$. Minding $0.01$ in (\ref{lower1}) we deduce from (\ref{upper1})  the inequality 
	$$\frac{8A\mathcal{C}(\Delta_\alpha)}{d} + \log(\frac{AX^{1/2}}{d})  + C +\frac{1}{d}\log|\NN_{K/\Q}(x-\alpha)| \geq Y $$
	where 
	$$C = \log(\mathcal{C}(\Delta_\alpha)) + 4\log|\Delta_\alpha| + \height(\alpha) + 1.04,$$
	$$Y = \max\{\dfrac{\pi X^{1/2}}{\mathcal{C}(\Delta)}, \dfrac{3}{\sqrt{5}}\log X - 9.78\}.$$
	
	We rewrite this as
	\begin{equation}
		\dfrac{8A\mathcal{C}(\Delta_\alpha)}{dY} + \dfrac{\log A + C}{Y} + \dfrac{\log(X^{1/2}/d)}{Y} + \dfrac{\log|\NN_{K/\Q}(x-\alpha)|}{dY} \geq 1.
		\label{inequality1}
	\end{equation}
	Note that $C > 3.11 > 0$, $\log A \geq 0 $ because $C \geq 4\log 7 + \frac{1}{4\sqrt{5}}\log7 -5.93 + 1.04 > 3.11$. Hence, we may replace $Y$ by $\frac{3}{\sqrt{5}}\log X -9.78$ in the second term of the left-hand side in  (\ref{inequality1}). Similarly, in the 1st term and 4th term we may replace $Y$ by $\pi X^{1/2}/\mathcal{C}(\Delta)$, and in the 3rd term we may replace $X^{1/2}/\mathcal{C}(\Delta)$ by $\pi^{-1}Y$. Notice that $d\geq \mathcal{C}(\Delta)$, we obtain 
	\begin{equation}
		\dfrac{8A\mathcal{C}(\Delta_\alpha)}{\pi X^{1/2}} + \dfrac{\log A + C}{\frac{3}{\sqrt{5}}\log X -9.78} + \dfrac{\log(\pi^{-1}Y)}{Y} + \dfrac{\log|\NN_{K/\Q}(x-\alpha)|}{\pi X^{1/2}} \geq 1.
		\label{inequality2}
	\end{equation}
	
	To obtain a lower bound of $\log|\NN_{K/\mathbb{Q}}(x-\alpha)|$, we will bound from above each of the three terms in its left-hand side. 
	
	From the results in \cite[Section 5.2 and Section 5.3]{bilu2017no}, we know that, when $X \geq 10^{15}$,
	\begin{equation}
		\log A \leq \frac{\log 2}{2}\frac{\log X}{\log\log X -c_1 - \log 2} + \log\log X,
		\label{logA}
	\end{equation}
	where $c_1 < 1.1713142.$
	
	\subsection{Bound the first term in (\ref{inequality2})}
	From above, easy to know that when $X \geq 10^{15}$, we have 
	$$\frac{\log(AX^{-1/2})}{\log X} \leq u_0(X),$$
	where
	$$u_0(X) = \frac{\log 2}{2}\frac{1}{\log\log X -c_1 - \log 2} + \frac{\log\log X}{\log X} - \frac{1}{2}$$
	which is decreasing for $X \geq 10^{15}.$ Hence
	$$\frac{\log(AX^{-1/2})}{\log X} \leq u_0(X) \leq u_0(10^{15}) \leq -0.1908,$$
	so
	$$\dfrac{8A\mathcal{C}(\Delta_\tau)}{\pi X^{1/2}} \leq \dfrac{8\mathcal{C}(\Delta_\tau)}{\pi} X^{-0.1908}\leq \dfrac{8}{\pi}\cdot 10^{15\cdot(-0.1908)} \leq 0.0035,$$
	since $X\geq \mathcal{C}(\Delta_\tau)^6\cdot 10^{15}$.
	
	\subsection{Bound the second term in (\ref{inequality2})}
	Obviously, by (\ref{logA})
	$$\dfrac{\log A + C}{\frac{3}{\sqrt{5}}\log X -9.78} \leq u_1(X)u_2(X),$$
	where 
	$$u_1(X) = \frac{\log 2}{2}\frac{1}{\log\log X -c_1 - \log 2} + \frac{\log\log X + C}{\log X},$$
	$$u_2(X) = (\frac{3}{\sqrt{5}} - \frac{9.78}{\log X})^{-1},$$
	which are decreasing for  $X \geq 10^{10}$.
	
	Since $X \geq e^{3.12}(\mathcal{C}(\Delta_\alpha)|\Delta_\alpha|^4e^{\height(\alpha)})^{3} = e^{3C}$, we have 
	$$\frac{\log\log X + C}{ \log X} \leq 0.6.$$
	Indeed, set $g(x) = \log x - 0.6x + C,$ which is decreasing for $x> 5/3$. Let $x_0 = 3C > 9.33 \geq 5/3$, since $C > 3.11$. Hence
	$$g(x) \leq g(x_0) = \log3 + \log C - 0.8C \leq  \log3 + \log(3.11) - 0.8 \cdot 3.11 <0.$$ 
	With this we have 
	$$u_1(X)u_2(X) \leq(\frac{\log 2}{2}\frac{1}{\log\log(10^{15}) -1.1713142 - \log 2} + 0.6)\cdot u_2(10^{15})< 0.7621.$$ 
	\subsection{Bound the third term in (\ref{inequality2})}
	For this term, we directly use the bound from \cite[subsection 5.5]{bilu2017no}
	$$\dfrac{\log(\pi^{-1}Y)}{Y} < 0.0672.$$
	
	\subsection{Summing up}
	We can combine the above estimates and bound $\dfrac{\log|\NN_{K/\mathbb{Q}}(x - \alpha)|}{\pi X^{1/2}}$ by 
	$$\dfrac{\log|\NN_{K/\mathbb{Q}}(x - \alpha)|}{\pi X^{1/2}} > 1 - (0.0035 + 0.7621 + 0.0672) = 0.1672,$$
	so 
	$$\log|\NN_{K/\mathbb{Q}}(x - \alpha)| > \dfrac{|\Delta|^{1/2}}{2}.$$
	
	\section{Proof of Theorem \ref{main theorem} (2) \label{section proof2}}
	
	As in the last section, we assume that $X = |\Delta| \geq 10^{15}$. By inequality~(\ref{lower1}), (\ref{lower2}) and Corollary~\ref{upperbound3} (2), we have 
	$$\dfrac{4A}{\mathcal{C}(\Delta)} + 2\log(\frac{AX^{1/2}}{\mathcal{C}(\Delta)})  + C +\dfrac{1}{\mathcal{C}(\Delta)}\log|\NN_{K/\Q}(x-1728)| \geq Y $$
	where 
	$$C = \height(1728) + \log 2 - 2.68 + 0.01= \log(3456) - 2.67 > 0,$$
	$$Y = \max\{\frac{\pi X^{1/2}}{\mathcal{C}(\Delta)}, \frac{3}{\sqrt{5}}\log X - 9.78\}.$$
	We rewrite this as
	\begin{equation*}
		\dfrac{4A}{\mathcal{C}(\Delta)Y} + \frac{2\log A + C}{Y} + \dfrac{2\log(X^{1/2}/\mathcal{C}(\Delta))}{Y} + \dfrac{\log|\NN_{K/\Q}(x-1728)|}{\mathcal{C}(\Delta)Y} \geq 1.
	\end{equation*}
	Hence,
	\begin{equation}
		\dfrac{4A}{\pi X^{1/2}} + \dfrac{2\log A + C}{\frac{3}{\sqrt{5}}\log X - 9.78} + \dfrac{2\log(\pi^{-1}Y)}{Y} + \dfrac{\log|\NN_{K/\Q}(x-1728)|}{\pi X^{1/2}} \geq 1.
	\end{equation}
	Using the similar method to estimate each term when $X \geq 10^{15}$, we have
	$$\dfrac{4A}{\pi X^{1/2}} < 0.0018,$$
	$$\dfrac{2\log A + C}{\frac{3}{\sqrt{5}}\log X - 9.78} < 0.7337,$$
	$$\dfrac{\log(\pi^{-1}Y)}{Y} < 0.0672,$$ 
	$$\dfrac{\log|\NN_{K/\Q}(x-1728)|}{\pi X^{1/2}} \geq 1-(0.0018+0.7337+ 2\cdot0.0672) = 0.1301.$$
	Hence,
	$$\log|\NN_{K/\Q}(x-1728)| \geq 0.1301\pi X^{1/2} \geq \frac{2|\Delta|^{1/2}}{5}.$$
	
	\section{Proof of Theorem \ref{main theorem} (3) \label{section proof3}}
	
	As before, we assume that $X = |\Delta| \geq 10^{15}$. By Proposition~\ref{lowerbound1}, Proposition~\ref{lowerbound2} and Corollary~\ref{upperbound3} (3), we have
	$$\frac{12A}{\mathcal{C}(\Delta)} + 3\log\frac{AX^{1/2}}{\mathcal{C}(\Delta)} - 3.76 + \frac{1}{\mathcal{C}(\Delta)}\log|\NN_{K/\Q}(x)| \geq Y,$$
	where
	\begin{equation*}
		Y = \max\{\frac{\pi X^{1/2}}{\mathcal{C}(\Delta)}, \frac{3}{\sqrt{5}}\log X - 9.78\},
	\end{equation*}
	
	We rewrite this as
	\begin{equation}
		\dfrac{12A/\mathcal{C}(\Delta)}{Y} + \dfrac{3\log A - 3.76}{Y} + \dfrac{3\log(X^{1/2}/\mathcal{C}(\Delta))}{Y} + \dfrac{\log|\NN_{K/\Q}(x)|/\mathcal{C}(\Delta)}{Y} \geq 1.
		\label{inequality4}
	\end{equation}
	Note that $3\log A-3.76 > 0$ because $A \geq \log X \geq \log(10^{15}) > 30$. Hence, we obtain 
	\begin{equation*}
		\frac{12A}{\pi X^{1/2}} + \frac{3\log A - 3.76}{\frac{3}{\sqrt{5}}\log X -9.78} + \frac{3\log(\pi^{-1}Y)}{Y} + \dfrac{\log|\NN_{K/\Q}(x)|}{\pi X^{1/2}} \geq 1.
		\label{inequality5}
	\end{equation*}
	
	From the results in \cite[Page 23 to Page 25]{bilu2017no}, we know that, when $X \geq 10^{15}$,
	$$AX^{-1/2} < 0.0014,$$
	$$\dfrac{3\log A - 3.76}{\frac{3}{\sqrt{5}}\log X -9.78} < 0.7734,$$
	$$\dfrac{\log(\pi^{-1}Y)}{Y} < 0.0672.$$
	We can combine the above estimates and bound $\dfrac{\log|\NN_{K/\Q}(x)|}{\pi X^{1/2}}$ by 
	$$\dfrac{\log|\NN_{K/Q}(x)|}{\pi X^{1/2}} > 1 - (12\pi^{-1}\cdot 0.0014 + 0.7734 + 3 \cdot 0.0672) > 0.019,$$
	so 
	$$\log|\NN_{K/\Q}(x)| > \frac{|\Delta|^{1/2}}{20}.$$
	
	\section{Proof of Corollary~\ref{bound for modular polynomials}}

	We know that the degree of $\Phi_m(X,Y)$ at $Y$ is $\deg_Y\Phi_m(X,Y) =\sigma_1(m)$. Assume that $\alpha = j(\tau)$, and $\{\gamma_1,\cdots, \gamma_{\sigma_1(m)}\}$ is a set of representatives of $\SL_2(\Z)\backslash D_m$. We set $\alpha_i = j(\gamma_i(\tau))$, which are also singular moduli. 
	
	Let $L = K(\alpha_1,\cdots, \alpha_{\sigma_1(m)})$ and $K_i = \Q(x, \alpha_i)$. We have $[K_i(\alpha): K] \leq \deg_Y\Phi(X,Y)= \sigma_1(m)$, and 
	\[\NN_{K_i(\alpha)/\Q}(x-\alpha_i) = \NN_{K_i/\Q}(x-\alpha_i)^{[K_i(\alpha):K_i]}.\]
	Then, when $\Delta$ is large enough, by Theorem~\ref{main theorem}, we have
	\begin{align*}
		\log|\NN_{K/\Q}(\Phi_m(x,\alpha))| & = \log|\NN_{L/\Q}(\prod\limits_{i=1}^{\sigma_1(m)}(x-\alpha_i))^{1/[L:K]}|\\
		& = \frac{1}{[L:K]}\sum\limits_{i=1}^{\sigma_1(m)}\log|\NN_{L/\Q}(x-\alpha_i)|\\
		& = \frac{[L:K_i]}{[L:K]}\sum\limits_{i=1}^{\sigma_1(m)}\log|\NN_{K_i/\Q}(x-\alpha_i)|\\
		&= \frac{[K_i(\alpha):K_i]}{[K_i(\alpha):K]}\sum\limits_{i=1}^{\sigma_1(m)}\log|\NN_{K_i/\Q}(x-\alpha_i)|\\
		& \geq \frac{1}{\sigma_1(m)}\sum\limits_{i=1}^{\sigma_1(m)}\log|\NN_{K_i/\Q}(x-\alpha_i)|\\
		& >\frac{1}{\sigma_1(m)}\sum\limits_{i=1}^{\sigma_1(m)}\frac{|\Delta|^{1/2}}{20}\\
		& = \frac{|\Delta|^{1/2}}{20}.
	\end{align*}

	\section*{Acknowledgements}
	The research of the author was supported by the China Scholarship Council. The author thanks his supervisors Yuri Bilu and Qing Liu for their supports, especially the helpful discussions and valuable suggestions from Yuri Bilu about this paper. He thanks Gabriel Dill for pointing out some  inaccuracies. Also the author would like to thank the referee for carefully reading this paper and useful suggestions.

	\address
\end{document}